\documentclass[12pt]{article}

\usepackage{amsmath}
\usepackage{amssymb}
\usepackage{amsfonts}
\usepackage{cite}
\usepackage{graphicx}
\usepackage{subfigure}
\usepackage{tikz}
\usepackage{hyperref}
\usepackage{cases}
\usepackage{lineno}

\date{}

\newtheorem{theorem}{Theorem}[section]
\newtheorem{lemma}{Lemma}[section]
\newtheorem{corollary}{Corollary}[section]

\newenvironment{proof}[1][Proof]{\begin{trivlist}
\item[\hskip \labelsep {\bfseries #1}]}{\end{trivlist}}
\newcommand{\qed}{\nobreak \ifvmode \relax \else
      \ifdim\lastskip<1.5em \hskip-\lastskip
      \hskip1.5em plus0em minus0.5em \fi \nobreak
      \vrule height0.75em width0.5em depth0.25em\fi}

\title{Distributed computation of homology using harmonics}
\author{Harish Chintakunta \and Hamid Krim}

\hypersetup{
    colorlinks, linkcolor={blue},
    citecolor={blue}, urlcolor={blue}
}

\begin{document}
\maketitle

\begin{abstract}
We present a distributed algorithm to compute the first homology of a simplicial complex. Such algorithms are very useful in topological analysis of sensor networks, such as its coverage properties. We employ spanning trees to compute a basis for algebraic 1-cycles, and then use harmonics to efficiently identify the contractible and homologous cycles. The computational complexity of the algorithm is $O(|P|^\omega)$, where $|P|$ is much smaller than the number of edges, and $\omega$ is the complexity order of matrix multiplication. For geometric graphs, we show using simulations that $|P|$ is very close to the first Betti number.
\end{abstract}


\section{Introduction}

Many tasks in sensor networks may be \cite{chintakunta2011topological} alternatively stated in topological terms. The tasks of detection and localization of coverage holes and worm holes are two such examples \cite{tahbaz2010distributed,chintakunta2011topological}. Many other tasks such as detecting jamming regions \cite{xu2006jamming} or detecting spatially correlated signals \cite{gupta2005efficient,chintakuntaCAMSAP2011} also come under the the context of computing certain topological features. Owing to the characteristics of limited power supplies hardware, it is preferable to have algorithms which are distributed and perform the tasks with no more information than is theoretically required.\\

The information which may be easily obtained in sensor networks, is for each node to know its neighboring nodes. Two nodes are neighboring nodes if they can communicate with each other. This is equivalent to having a distributed representation of the communication graph. With limited communication between the nodes, we may also obtain the higher order cliques in the graph \cite{muhammad2007decentralized} to form a simplicial complex. A $k-$clique is complete subgraph with $k$ nodes.  Such combinatorial information is sufficient to compute the topological invariants, and is the subject of Algebraic Topology. \\

In addition, algebraic topology has seen useful applications in topological analysis of data \cite{carlsson2009topology,gamble2010exploring,heo2012topological,jiang2011statistical}, where the data is available as a point cloud and the topology of the underlying manifold is of interest. A simplicial complex is first built on the point cloud data, usually using some geometric criteria, followed by computation of topological invariants such as homology spaces and persistent homology. The computations in this scenario need not necessarily be distributed, but parallel and distributed algorithms can greatly speed up the process.\\

The application of homology spaces to the coverage problem was first introduced in \cite{deSilvaGhrist}. Given a set of ``fence'' edges, the work in \cite{deSilvaGhrist} provides a necessary and sufficient condition to verify the coverage inside the fence, and recent work in \cite{dlotko2012distributed} provides a distributed algorithm to perform this verification. The work in \cite{dlotko2012distributed} also computes some specific generators of certain homology space, which may be used to select a subset of sensors which could still provide the required coverage.\\

When coverage cannot be guaranteed, \cite{ghrist2005coverage} gives a criterion using persistent homology to guarantee the existence of holes. In this scenario, counting, enumerating and localizing the coverage holes can prove very useful in planning counter measures such as additional deployment on sensors. Counting the number of holes requires computing the first Betti number (the rank of the 1st homology space), and enumerating and localizing requires computing an optimal (in terms of length) homology generating set. Centralized algorithms for computing an optimal generating set have been developed recently \cite{busaryev2012annotating}. \\

Identifying individual coverage holes is also important in surveillance and tracking applications. Given specific generators of homology (cycles in the graph) and limited technological resources, it is possible to partially track multiple agents moving in the observation field  even in the presence of coverage holes \cite{yu2010probabilistic,yu2008tracking,5508237}.\\

The work in \cite{muhammad2007decentralized} is, according to our knowledge, the first work towards decentralized computation of homology in sensor networks. The algorithm in \cite{muhammad2007decentralized} computes a single generator for homology with coefficients in  $\mathbb{R}$, and therefore, cannot generate the entire homology in case of more than one hole. The generator is a harmonic computed in a manner similar to that described in this paper. For arbitrary number of coverage holes, they propose a spectral decomposition of the first Laplacian using the distributed algorithm presented in \cite{kempe2004decentralized}. The communication complexity per node of the algorithm in \cite{kempe2004decentralized} is $O(k^3)$, for computing $k$ principal eigenvectors. Therefore, a complete spectral decomposition of the Laplacian using this method is burdensome in sensor networks, where distributed algorithms with lower complexity are needed. \\

Localization of coverage holes is done in \cite{tahbaz2010distributed} by minimizing the $l_1$ norm of cycles homologous to a harmonic, and when the holes are sufficiently apart, can enumerate them. They employ sub-gradient methods to minimize the $l_1$ norm distributively which are slow but are guaranteed to converge. The approach used for localization in \cite{5508237,chintakunta2011topological} is to successively partition the network into two parts, and detect the presence of holes in each. Since detection of holes (or equivalently, the triviality of the homology space) is a simpler problem than that computing the homology itself, this approach leads to significant improvement in the complexity. Both these methods however, do not guarantee a proper enumeration, and hence, cannot identify individual coverage holes.\\

A distributed algorithm which can compute homology in $\mathbb{Z}_2$ coefficients will efficiently enumerate the coverage holes and help localize them. The algebraic cycles with coefficients in $\mathbb{Z}_2$, have a straightforward interpretation as cycles in the graph. It is for this reason, homology with coefficients in $\mathbb{Z}_2$ is most appropriate for networks. This paper presents a completely distributed algorithm to compute the first homology space. \\


In addition to their applications in sensor networks, distributed algorithms are very useful for processing in computer clusters which have a similar architecture. There have been some efforts recently to parallelize the homology computation \cite{boltcheva2010constructive,lewismulticore}, which employ parallel processors with shared memory. Computing persistent homology is another way of computing the homology with algorithms known in matrix multiplication time, where the size of the matrix is equal to the number of simplices in the complex \cite{milosavljevic2010zigzag}.  The fastest known algorithms are centralized \cite{viditMorseTheoryPersistence,wilkersonSimplifying2013}, which first reduce the given complex into a smaller complex while preserving the homology, and then compute the homology of the resulting complex. \\

The inspiration for the work presented here stems from the following simple question: Given a cycle in a graph (more specifically, in the 1-skeleton of a simplicial complex), is there a simple distributed methodology to test if the cycle is contractible? The answer turns out be, yes! This ability along with the fact that a basis for 1-cycles may easily be computed distributively, leads to the algorithm presented here.  We give a brief summary of our algorithm below. The meaning of the notations in the summary will be clear form the context, and they are formally defined in Section \ref{sec:Preliminaries}.

\subsection{summary of algorithm}
Our goal is to compute a maximal (in cardinality) set $H$ of 1-cycles $\{c_i\}$, such that the corresponding equivalent classes $\{ [c_i] \}$ are linearly independent elements in homology, thus forming a basis. We say then the 1-cycles in $\{c_i\}$ are homology generating cycles and that the set $\{c_i\}$ generates the first homology.  We first obtain a basis $Z$ for the 1-cycles using a spanning tree. This approach for computing $Z$ is well known and is extensively used in the literature \cite{stillwell1980classical,erickson2005greedy,busaryev2012annotating}. Fast distributed algorithms for computing minimum spanning trees for general weighted graphs are also well known \cite{gallager1983distributed}. When all the edges have equal weight, finding a spanning tree becomes particulary simple as we describe in Section \ref{subsec:computing_spanning_tree}.  \\

We then use a harmonic on the 1-skeleton to identify homologous cycles and select a subset $P \subseteq Z$ of representative cycles which contains the required set $H$. The algorithm presented in \cite{muhammad2006control} distributively computes a harmonic by solving a dynamic system. We provide a simple  methodology to distributively select the parameters to guarantee convergence, and show that the algorithm converges exponentially. \\

The final stage of selecting a maximal linearly independent (in homology) subset $H$  of $P$ is performed at the root node of the spanning tree. The set $P$ is usually much smaller than $Z$, and for geometric graphs, we show using simulations, that the size of $P$ is very close to the first Betti number. The selection of the subset $H$ is done by column reducing a matrix, which as observed in \cite{busaryev2012annotating}, may be performed in matrix multiplication time.

\subsection{organization}
The remainder of the paper is organized as follows. In Section \ref{sec:Preliminaries}, we introduce some basic concepts from homology theory and notations used in this paper. Sections \ref{sec:computing_harmonics} and \ref{sec:homology_generating_cycles} describe the centralized version of the algorithm. We describe the computation of harmoincs in Section \ref{sec:computing_harmonics}, and computation of homology generating set in Section \ref{sec:homology_generating_cycles}. Section \ref{sec:distributed_computation} presents the details of distributed computation of the algorithm. We analyze the complexity of both centralized and distributed versions in Section \ref{sec:complexity}, and finally, conclude in Section \ref{sec:conclusion}.

\section{Preliminaries}
\label{sec:Preliminaries}
In this section, we briefly introduce some concepts of homology theory and some notations. We refer the readers to \cite{hatcheralgebraic}, for an introduction to algebraic topology , and \cite{de2007homological,chintakunta2011topological,muhammad2006control}, for applications of this theory in analysis of coverage properties in sensor networks.   \\

Given a simplicial complex $K$, the 1-skeleton is a graph $G = (V,E)$, which is a subcomplex of $K$ with nodes (0-simplices) $V$ and edges (1-simplices) $E$. The number of nodes is denoted by $N$ and number of edges by $|E|$. We assume that $G$ is connected. \\

Abstract vector spaces $C_0,C_1$ and $C_2$ are constructed using 0,1 and 2 simplices respectively as basis elements. An $i-$simplex has an arbitrary, but fixed, binary orientation given by the ordering on its vertex set. Given a $i-$simplex, the simplex with opposite orientation is its additive inverse in $C_i$. The linear operators $\partial_i:C_i \rightarrow C_{i-1}$, denoted as boundary operators, capture the combinatorial structure of the simplicial complex and are given as $\partial_i(v_{j_0},\ldots,v_{j_i}) = \sum_{l=0}^{i}{(-1)^{l}(v_{j_0},\ldots,\hat{v}_{j_l},\ldots,v_{j_i})}$, where $\hat{v}_{j_l}$ implies that $v_{j_l}$ is removed.\\

The null space of the first boundary operator $ker(\partial_1)$ is the space of 1-cycles. The image of the second boundary operator $Img(\partial_2)$ is the space 1-boundaries. The first homology space $H_1(K)$ is defined as the quotient space $ker(\partial_1)/Img(\partial_2)$. As we are using field coefficients, any torsion in the integer homology of $K$ will be lost. Many complexes of interest, such as Rips complex of a geometric graph on a plane  used to model sensor networks, can be shown to have to have no torsion \cite{chambers2010vietoris}.\\

Two cycles $c_1,c_2 \in ker(\partial_1)$ are said to be homologous to each other if their difference is a boundary, i.e. $c_1 - c_2 \in Img(\partial_2)$.  Elements in $H_1(K)$ are equivalent sets of cycles which are equivalent to each other. The set of homologous cycles to a cycle $c$ is denoted as $[c]$. Consider a set of cycles $\{c_i\}$ such that the corresponding set of equivalent classes $\{[c_i]\}$ forms a basis for $H_1(K)$. We call such a set $\{c_i\}$ a homology generating set, denoted by $H$. Our goal is to compute a homology generating set $H$. \\

The first order combinatorial Laplacian $L_1:C_1\rightarrow C_1$, is defined as
\begin{equation}
L_1 = \partial_2\partial_2^T + \partial_1^T\partial_1
\end{equation}
and the space of 1-harmonics \footnote[2]{which we refer to as harmonics for simplicity} is defined as the null space of $L_1$. We also denote by $L_1$, the matrix representation of the linear operator in the standard basis for $C_1$, and by $\partial_1$ and $\partial_2$, the matrix representation of the boundary operators in the corresponding standard bases.  We refer the reader to Section III in \cite{muhammad2006control} for a formula for the elements of $L_1$ and examples of harmonics. We present this formula in Appendix \ref{app:formulaForL1} for readers' convenience.  In what follows, we refer to the first homology simply as homology.


\section{Computing harmonics}
\label{sec:computing_harmonics}
Consider the following dynamic system:
\begin{equation}
\frac{d y(t)}{dt} = - L_1 y(t) \nonumber
\end{equation}
Note that the stable point of the above dynamic system is a harmonic. It is shown in \cite{muhammad2006control} that the above dynamic system converges for any initial point $y(0)$. A discrete version of the above system is given by:

\begin{equation}
\label{equ:randomHarmonicIteration}
y^{k+1} = y^{k} - \delta L_1 y^k
\end{equation}

Here, we derive the sufficient conditions for the range of  $\delta$ to guarantee convergence. We also show that, 1)  under these conditions, Iteration \ref{equ:randomHarmonicIteration} has a unique convergence point, 2) that it converges exponentially, and 3) derive the convergence rate.\\

Let $y^0$ be a random vector of dimension $E$, where the elements are generated independently from a uniform distribution on the interval $[-0.5,0.5]$. Since $L_1$ is a diagonalizable matrix,  vector $y^0$ may be expressed as a linear combination of eigenvectors of $L_1$. Let $0<\lambda_1 \le \lambda_2 \le \ldots \lambda_m$ be positive eigenvalues of $L_1$, and let $y^0 = \sum_i{\alpha_i v_i}$, where $v_i$ are orthonormal eigenvectors of $L_1$. Let $K$ be a matrix with column space  equal to the null space of $L_1$. The projection of $y^0$ onto the null space of $L_1$ is equal to $KK^Ty^0$.  Using Equation \ref{equ:randomHarmonicIteration}
\begin{equation*}
y^1 = KK^Ty^0 +  \sum_i{(1-\delta \lambda_i)\alpha_i v_i}
\end{equation*}
and in general,
\begin{equation}
\label{iteration_equation}
y^k = KK^Ty^0 + \sum_i{(1-\delta \lambda_i)^k\alpha_i v_i}
\end{equation}
The sequence $\{y^k\}$ converges if and only if $|1-\delta \lambda_i| < 1$ for all $i$, or equivalently,
\begin{equation}
\label{equ:deltaValues}
0 < \delta < 2/\lambda_i, \quad \forall i
\end{equation}
Given a starting vector $y^0$, and a scalar $\delta$ satisfying inequality (\ref{equ:deltaValues}), the sequence $\{y^k\}$ converges uniquely to $KK^Ty^0$ and the convergence rate is dominated by the smallest non-zero eigen value. The above discussion is summarized in the following lemma.

\begin{theorem}
\label{theo:harmonicConvergence}
Let $y^0$ be an initial random vector for Iteration  (\ref{equ:randomHarmonicIteration}). Let $K$ be a matrix with column space  equal to the null space of $L_1$. Then the iteration converges to $y^\infty =  KK^Ty^0$ if and only if $\delta$ satisfies inequality (\ref{equ:deltaValues}).

Further, when inequality ($\ref{equ:deltaValues}$) is satisfied, the iteration converges exponentially with rate $1-\delta\lambda_1$.
\end{theorem}

We may easily estimate the spectral radius of the matrix using a well known relationship, $\|A\|_2 \le  \sqrt{\|A\|_1 \|A\|_\infty}$, for any finite matrix $A$. Since $L_1$ is symmetric, we have $\|L_1\|_1 = \|L_1\|_\infty$ and therefore $\|L_1\|_2 \le \|L_1\|_1 $. The value $\delta = \frac{1}{\|L_1\|_1}$ satisfies inequality (\ref{equ:deltaValues}). The inequality $\|L_1\|_1 < \sqrt{n}\|L_1\|_2$  ensures  that the convergence rate is not affected severely by approximating the spectral radius with $\|L_1\|_1$.  We discuss the distributed computation of $\|L_1\|_1$ in Section \ref{sec:distributed_computation}.

\section{Homology generating cycles}
\label{sec:homology_generating_cycles}
We start with the assumption that we have a spanning tree $T = (V,E_T)$, $E_T \subseteq E$. A simple distributed algorithm to compute $T$ is discussed in Section \ref{sec:distributed_computation}. We select a node $v_r$, and call it the root node. Given a path $p=(e_{i_1} \cdots e_{i_k})$ in $G$, the corresponding 1-chain $\pi(p)$ is $\pi(p) = \sum_{j=1}^{k}{\alpha_j\sigma_j}$, where $\sigma_j$ is the standard basis element with the support set equal to the incident nodes of $e_{i_j}$, and $\alpha_j = 1$ if $\sigma_j$ and $e_{i_j}$ have the same orientation or $\alpha_j = -1$ otherwise. We define the integral of a harmonic $y$ on a path $p$ (or on the 1-chain $\pi(p)$) to be the dot product $\left<y,\pi(p)\right>$. \\

For an edge $e = (v_1,v_2) \in E \setminus E_T$, let $p_1$ be a path in $T$ joining $v_r$ to $v_1$, and $p_2$ be a path in $T$ joining $v_2$ to $v_r$. Then the path $(p_1,e,p_2)$ is a cycle in $G$, and denote by $\gamma(T,e)$, the corresponding 1-cycle. The set of 1-cycles $Z$ defined as $Z = \cup_{e\in E\setminus E_T}{\gamma(T,e)}$ forms a basis for all 1-cycles. The reader may refer to \cite{busaryev2012annotating} for a simple proof. Let $P$ be a maximal (in cardinality) subset $P \subseteq Z$ of non-contractible cycles, such that no two cycles in $P$ are homologous to each other. Since homology generating cycles are non-contractible, there exists a subset $H \subseteq P$ which generates the homology.\\

\subsection{Identifying contractible cycles}
We turn our attention now to deriving a sufficient condition to identify contractible cycles in $Z$.

\begin{lemma}
\label{lem:harmonics_orthogonal_to_boundaries}
Let $y$ be a harmonic and $c \in B$ a boundary. Then, $\left< y, c \right> = 0 $
\end{lemma}

\begin{proof}
\begin{eqnarray*}
y \in ker(L_1)  &\Rightarrow& L_1y = 0 \\
&\Rightarrow&  y^TL_1y=0 \\
&\Rightarrow&  y^T\partial_2\partial_2^T + y^T\partial_1^T\partial_1y^T = 0\\
&\Rightarrow& \parallel \partial_2^Ty \parallel^2 + \parallel \partial_1y \parallel^2 = 0 \\
&\Rightarrow& \partial_2^Ty = 0 \mbox{ and } \partial_1y = 0 \\
\end{eqnarray*}
and
\begin{eqnarray*}
\partial_2^Ty = 0  &\Rightarrow&  b^T \partial_2^Ty = 0, \qquad \forall b \in C_2 \\
&\Rightarrow& y^T \partial_2 b = 0 \qquad \forall b \in C_2 \\
&\Rightarrow& y^T c = 0 \qquad \forall c \in  B \\
&\Rightarrow& \left< y,c \right> = 0 \qquad \forall c \in B  \qquad  \blacksquare
\end{eqnarray*}
\end{proof}

The above lemma implies that all harmonics integrate to zero on (or are orthogonal to) contractible cycles. The question now is when does $\left<y,c\right>=0$ imply $c$ is contractible. To answer this, we will first look at the set of harmonics which are orthogonal to non-contractible cycles in $Z$.\\

All the possible cycles we consider are obtained from the tree, and therefore, as a vector of coefficients, the elements in $c$ are from the set $\{-1,0,1\}$. As a result, we need to consider only finitely many vectors we consider. For a given cycle $c$, the set of harmonics which are orthogonal to $c$ is the intersection of the hyperplane $c^\perp$ with $ker(L_1)$, the space of harmonics. Therefore, the set of harmonics which can possibly be orthogonal to at-least one of the non-contractible cycles is given by

\begin{equation}
\label{equ:orthogonalSet}
\mathcal{S} = \bigcup_{c\not\in B}{c^\perp \cap ker(L_1)}
\end{equation}

\begin{lemma}
\label{lem:orthogonalHarmonicsDimension}
Let $c$ be a non-contractible cycle, and $b_1$ be the first Betti number. The dimension of the set $c^\perp \cap ker(L_1)$ is strictly less than $b_1$ where $b_1$ is the dimension of $ker(L_1)$.
\end{lemma}

\begin{proof}
We prove the above statement by finding a harmonic which is not orthogonal to $c$. \\

The set of cycles $Z$ may be decomposed as $Z = B \oplus H_1$. The cycle $c \not\in  B$ may be expressed as $c = \alpha c_b + \beta c_h$ with $\beta \ne 0$. Let $c_h$ also denote the coefficient vector (of length $b_1$) expressed in some basis $\mathcal{B}_{H_1}$ for $H_1$. Let $K$ be a square matrix of size $b_1$, where each column represents a basis element of harmonics expressed in terms of elements in $\mathcal{B}_{H_1}$. \\

For a given cycle $c$, let $\hat{y} = KK^Tc_h$ be the projection onto the harmonic space. The dot product of $\hat{y}$ with $c$ gives
\begin{eqnarray*}
\hat{y}^Tc  &=& \hat{y}^T\left(\alpha c_b + \beta c_h \right) \\
&=& \beta
\end{eqnarray*}

$\blacksquare$
\end{proof}

A direct consequence of Lemma \ref{lem:orthogonalHarmonicsDimension} is that the the set $c^\perp \cap ker(L_1)$ has a measure zero. The set $\mathcal{S}$ given in (\ref{equ:orthogonalSet}) is a finite union of measure zero sets, and therefore has measure zero. This means that the statement ``$\left<y,c\right>=0 \Rightarrow c \in B$'' is false on a set of measure zero!\\

As stated in Theorem \ref{theo:harmonicConvergence}, our process of computing harmonics is equivalent to projecting a random vector on to the space of all harmonics. The above discussion leads to the following important theorem.

\begin{theorem}
\label{theo:contractible_cycles}
Let $c\in Z$ be a cycle in $Z$,  and let $y$ be a harmonic generated using iteration (\ref{equ:randomHarmonicIteration}). Then $c\in B \Leftrightarrow \left<y,c\right>=0$, with probability 1.
\end{theorem}

Theorem \ref{theo:contractible_cycles} allows us to easily identify all the non-contractible cycles in $Z$. Note that for $c_1,c_2 \in Z$, by definition, $c_1$ is homologous to $c_2$ if and only if either $c_1+c_2 \in B$ or $c_1 - c_2 \in B$. Further, the integration of a harmonic on cycles is a linear process. These facts, along with Theorem \ref{theo:contractible_cycles} results in the following corollary.

\begin{corollary}
\label{cor:homologous_cycles}
Let $c_1,c_2 \in Z$ be cycles in $Z$. Then $c_1$ is homologous to $c_2$ if and only if $|\left<y,c_1\right>| = |\left<y,c_2\right>|$, with probability 1.
\end{corollary}

Corollary \ref{cor:homologous_cycles} enables the identification of homologous cycles, which we use to partition the cycles in $Z$ into subsets of homologous non-contractible cycles.  We then pick one cycle from each subset in this partition to form the set $P$. We use the symbol $P$ to also denote an $E\times m$ matrix, where $i^{th}$ column is the coordinate vector of $i^{th}$ cycle in set $P$, expressed in the standard basis for $C_1$.  \\

Note that the equivalent classes of cycles in $P$ are not necessarily linearly independent in homology. However, the cardinality of $P$ is much less than that of $Z$, and as shown in Figure \ref{fig:linearExcessCycles}, is very close to the first Betti number for geometric graphs.

\begin{figure}[!ht]
\centering
\includegraphics[width=0.5\textwidth]{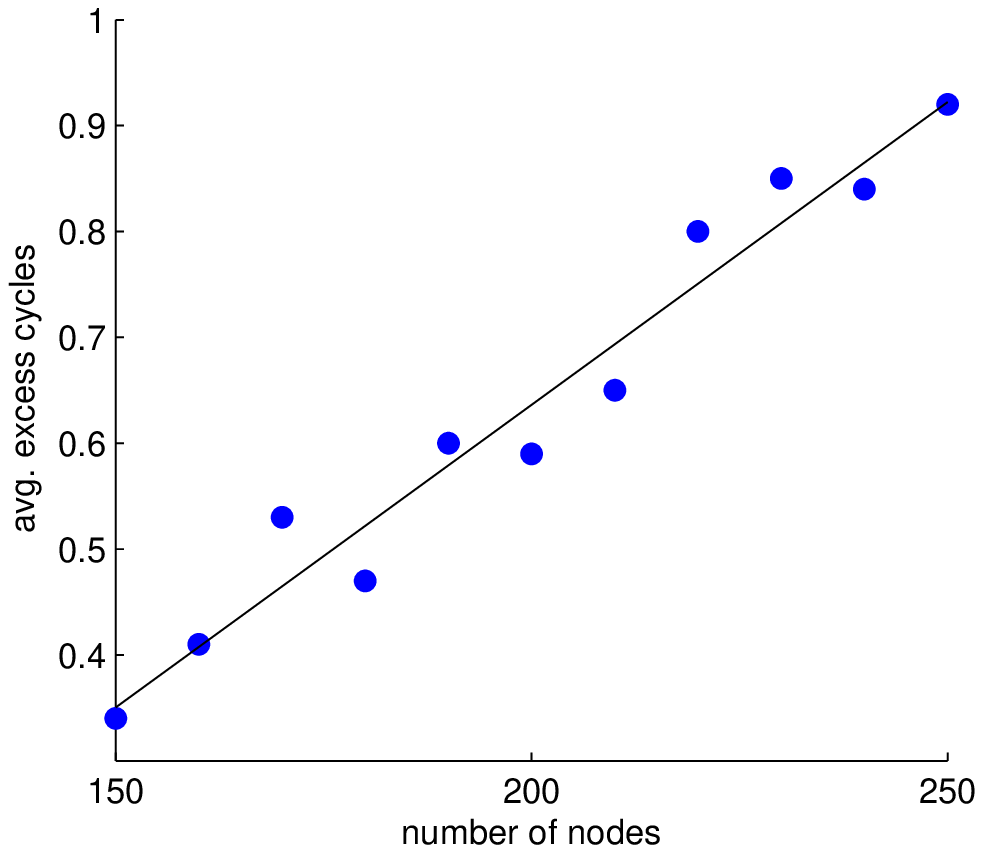}
\caption{}
\label{fig:linearExcessCycles}
\caption{figure shows the average number of excess cycles ($|P|-b_1$) remaining after pruning the tree and selecting the set $|P|$. on geometric graphs. The number of excess cycles increases linearly with the nodes, and is small compared to the Betti numbers.}
\end{figure}

\subsection{Selecting homology generating set}
\label{subsec:homology_generating_subset}
In this section, we discuss the selection of a minimal subset $H$ of $P$, which generates the homology. \\

Given an arbitrary ordering on cycles in $P$, computing the subset $H$ is equivalent to identifying all the cycles in $P$ such that the difference between each of these cycles and some linear combination of the others is contractible. In \cite{busaryev2012annotating}, for example,  this is accomplished  by  column reducing the matrix $\left[ \partial_2 \mbox{  } Z \right]$. If $r$ is the rank of $\partial_2$, then the last $b_1$ non-zero columns form the matrix $H$. The identification mentioned above is performed directly by using the range space of the $\partial_2$ operator. On the other hand, we posit that harmonics provide a very efficient way to identify these contractible cycles and thereby improving complexity. Further as shown in Section \ref{subsec:distributed_computing_harmonics} and \ref{sec:complexity}, harmonics can be computed distributively with low complexity. We now describe the application of harmonics to compute the homology generating set $H$.\\

We compute $m$ harmonics $y_i, i=1\ldots m$, using Iteration (\ref{equ:randomHarmonicIteration}), and stack them in the matrix $Y = [y_1,y_2,\ldots,y_m]$. Since the space of harmonics is isomorphic to the first homology, and the harmonics are generated randomly, the rank of $Y$ is equal to the first Betti number with probability 1. Also, since the cycles in $P$ are linearly independent, the rank of the matrix $R=Y^T P $ is also equal to the first Betti number. Note that computing the matrix $R$ is equivalent to integrating the harmonics on each cycle in $P$, with $R_{ij}$ equal to the integral of $i^{th}$ harmonic on the $j^{th}$ cycle. Further, consider a set of cycles $\{c_i\} \in P, i = i_1,\ldots,i_m$ whose cosets are linearly dependent in homology, i.e., $\sum_{i_1}^{i_m}{\alpha_i [c_i]}=0$, not all $\alpha_i=0$. Then, it follows from Lemma \ref{lem:harmonics_orthogonal_to_boundaries}, that $\sum_{i_1}^{i_m}{\alpha_i[R_i]}=0$, where $R_i$ is the $i^{th}$ column of $R$. Therefore, the cycles in $P$ corresponding to non-zero columns after reducing $R$, form the set $H$.\\

As observed in \cite{busaryev2012annotating}, this reduction may be performed in matrix multiplication time $O(|R|^\omega)$ \cite{ibarra1982generalization,jeannerod2006lsp}, where $|R|$ is the size of the matrix $R$, and $\omega$ is approximately 2.4.

\begin{figure}[!p]
\centering
\subfigure[]{
\includegraphics[width=0.2\textwidth]{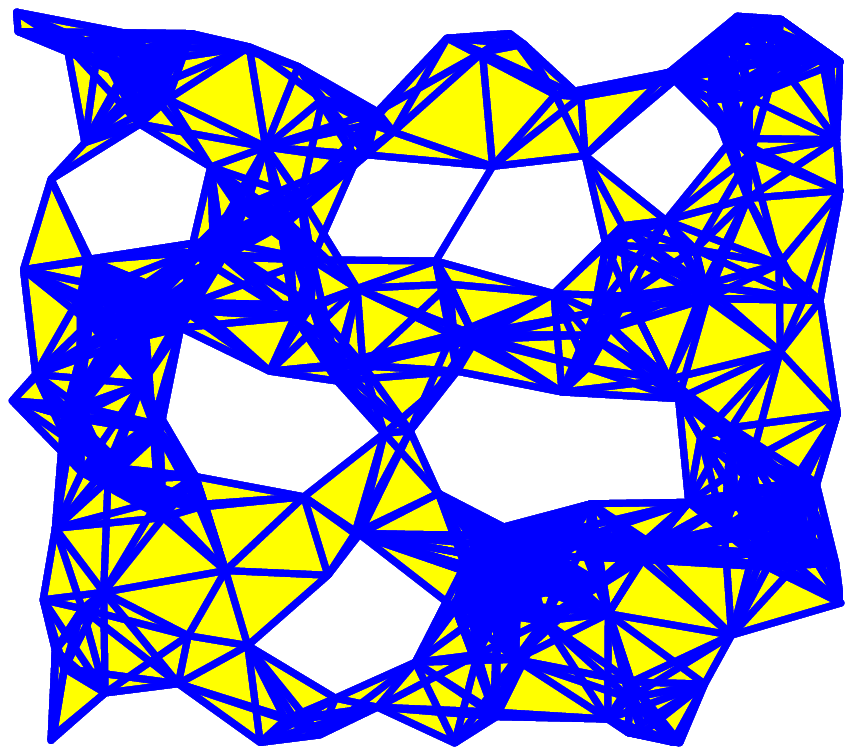}
\label{subfig:graph_filled_in}
}
\quad \quad \quad \quad
\subfigure[]{
\includegraphics[width=0.2\textwidth]{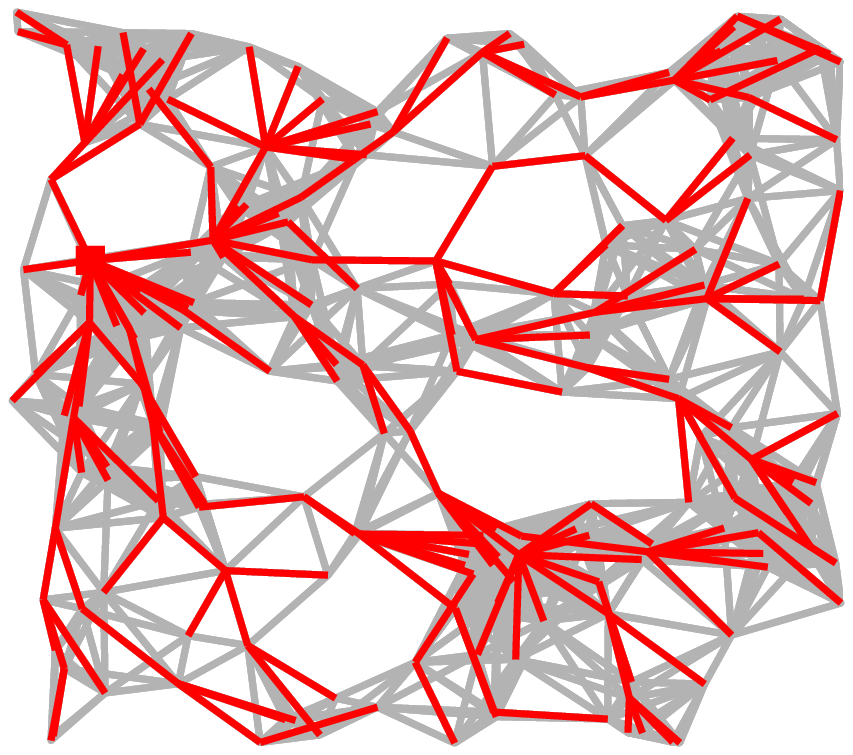}
\label{subfig:graph_spanning_tree}
}
\subfigure[]{
\includegraphics[width=0.2\textwidth]{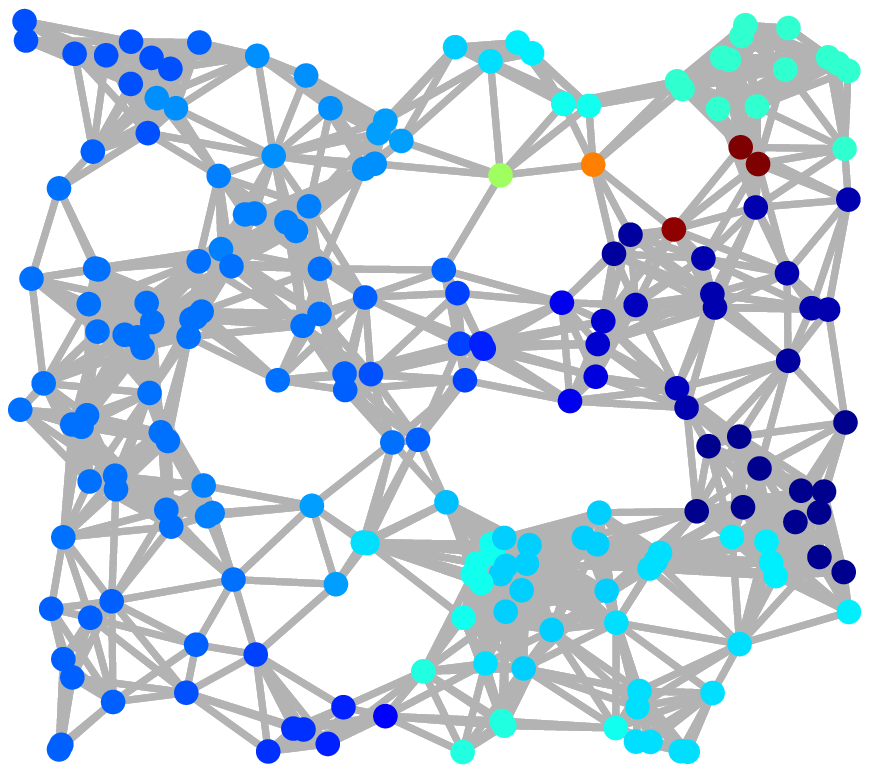}
\label{subfig:graph_harmonic_integral}
}
\quad \quad \quad \quad
\subfigure[]{
\includegraphics[width=0.2\textwidth]{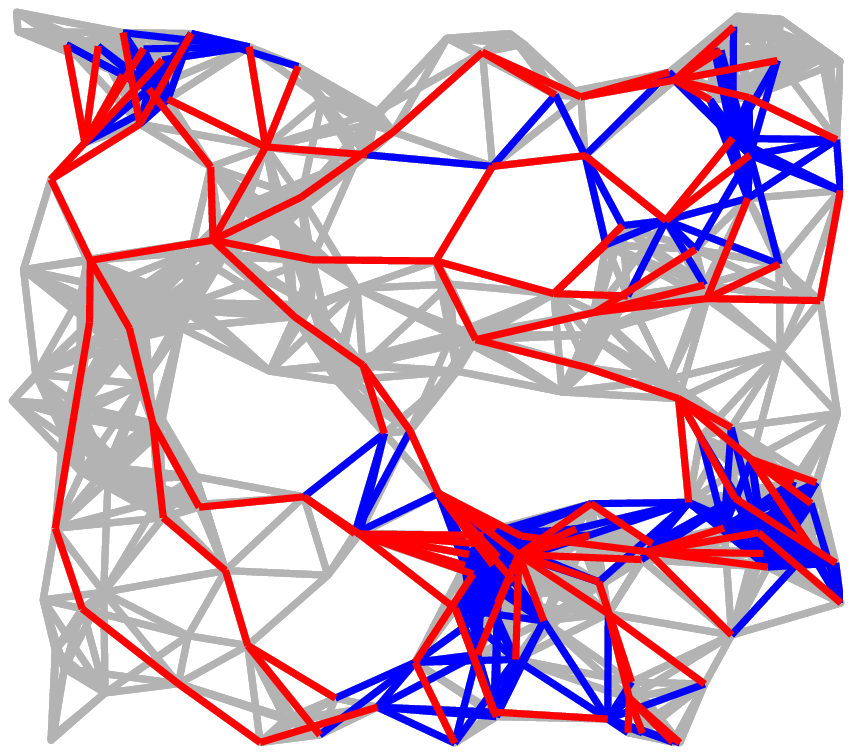}
\label{subfig:graph_pruned}
}
\quad \quad \quad \quad
\subfigure[]{
\includegraphics[width=0.2\textwidth]{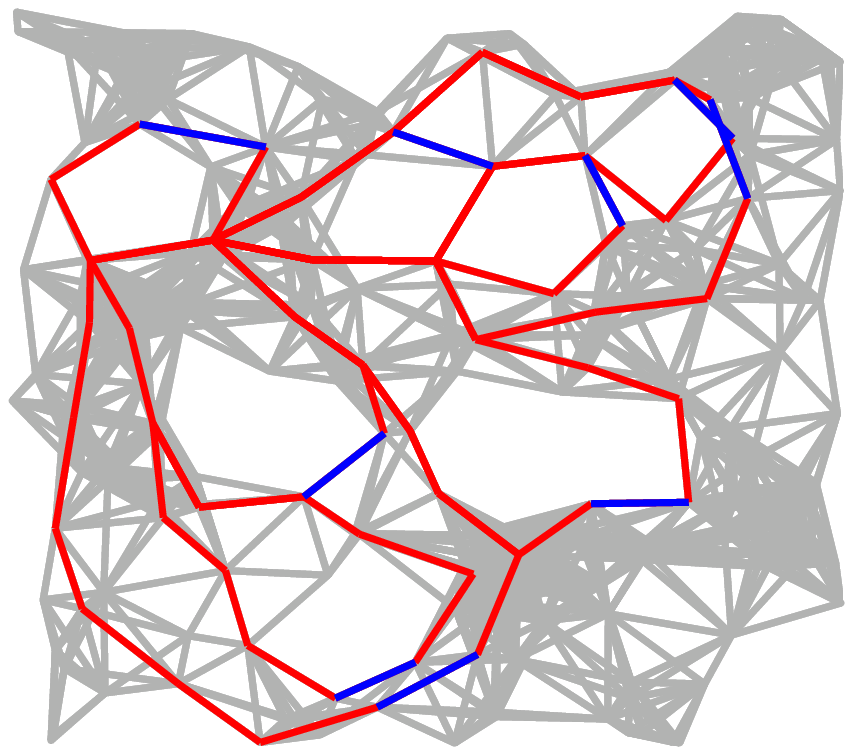}
\label{subfig:graph_pruned_clustered}
}
\quad \quad \quad \quad
\subfigure[]{
\includegraphics[width=0.2\textwidth]{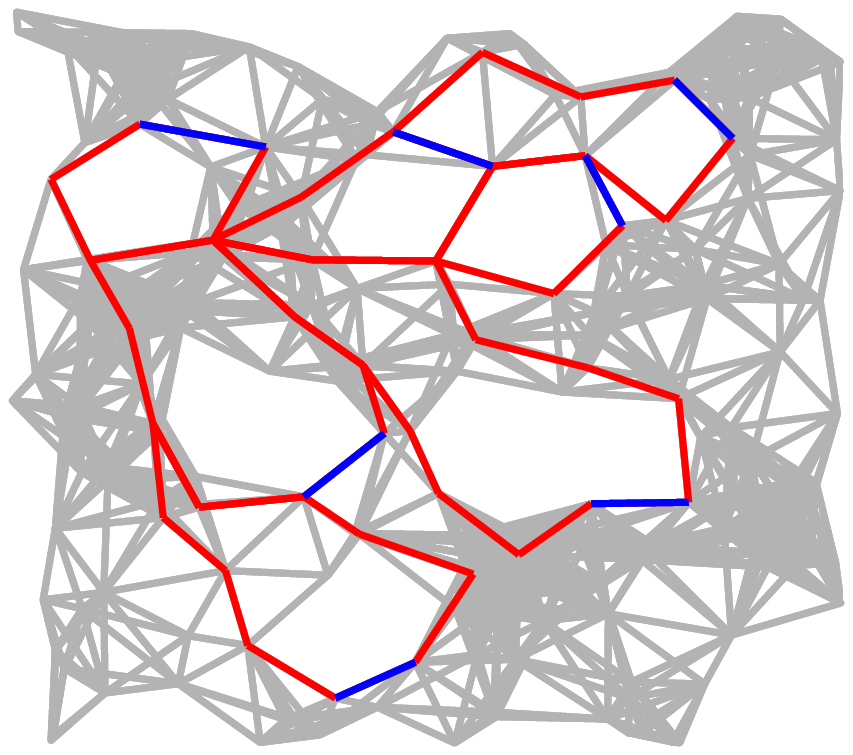}
\label{subfig:graph_homology_generators}
}
\caption[Illustration of distributed homology computation]{(a) graph with triangles ``filled in'', (b) spanning tree, (c) integral of harmonic on the tree (blue(low) to red(high) values), (d) edges corresponding to contractible cycles removed, (e) cycles in the set $P$ and (f) homology generating cycles, $H$}
\end{figure}

\section{Distributed computation}
\label{sec:distributed_computation}
Special attention is given to the communication complexity between processing units as the cost of communication between units is in general much higher than local processing.

\subsection{Computing harmonics}
\label{subsec:distributed_computing_harmonics}
Harmonics may be computed distributively using Iteration (\ref{equ:randomHarmonicIteration}), repeated here for convenience:
\begin{equation}
y^{k+1} = y^{k} - \delta L_1 y^k \nonumber
\end{equation}
In the distributed setting, each edge may be assigned a separate processing unit, and we assume that the architecture allows a processing unit to communicate simultaneously with units corresponding to adjacent edges.  In the case of sensor networks, this processing unit may be one of the nodes adjacent to the edge. We refer the reader to \cite{muhammad2007decentralized} for finer details of protocols required for such an emulation of an edge on one of the incident nodes. It suffices to say here that it requires at most 2 hops in the network for two incident edges to communicate. Note that $L_1$ is an $|E|\times |E|$ matrix, and its contents are given according to  formula given in Appendix \ref{app:formulaForL1}. The $i^{th}$ row of $L_1$ is stored locally in the processing unit corresponding to edge $e_i$. $L_1$ is generally sparse (and definitely so for geometric graphs modeling sensor networks), and therefore the memory required for storing a row locally is small. \\

As discussed in Section \ref{sec:computing_harmonics}, a choice of $\delta = 1/\|L_1\|_1$ guarantees the convergence of Iteration (\ref{equ:randomHarmonicIteration}). The ${l}_1$ norm for the Laplacian is given as $\|L_1\|_1 = \max_j{\sum_i{ |\{L_1\}_{ij}| }}$, the maximum absolute sum of its columns. Since $L_1$ is symmetric, it is also equal to the maximum absolute sum of its rows. Since the row $i$ is stored locally, its absolute sum can be computed locally. A simple gossip algorithm may be used to compute the maximum of these local sums. Table \ref{Tab:max} presents one such algorithm used in \cite{chintakunta2011topological}. The worst case in complexity occurs when each node, with value $x(i)$, discovers all the values greater than $x(i)$ in ascending order. When this happens, node $i$ will broadcast all the values greater than $x(i)$ and total number of broadcasts is equal to $n(n-1)/2$. The average number of broadcasts per node will be equal to $(n-1)/2$. However, this is a very loose upper bound and in practice, the number of broadcasts will be much smaller. \\

For the multiplication $y' = L_1 y^{k}$ in each iteration, the $i^{th}$ element is given as $y'(i) = \sum_j{ \{L_1\}_{ij}y^k(j)} = \sum_{j \in \mathcal{N}_i}{\{L_1\}_{ij}y^k(j)} $. The second equality results from the fact that the non-zero elements in the $i^{th}$ row of $L_1$ are those corresponding to adjacent edges. Therefore, this multiplication is performed simply by each processing unit/node broadcasting its value to its neighbors. The number of such broadcasts required is equal to the number of iterations required until convergence to a required precision. It is shown in Section \ref{sec:computing_harmonics} that the convergence is exponential with rate $1-\delta\lambda_1$. Specifically, the number of iterations required for a precision $\epsilon$ is equal to $log(\epsilon)/log(1-\delta\lambda_1)$.

We also see from simulation results shown in Figure \ref{fig:linearIncreaseInIterations}, that the number of iterations grows at-most linearly with  the number of nodes in geometric graphs.

\begin{table}
\centering
\begin{tabular}{l}
\hline\\
Distributed algorithm for computing $\max(x)$\\
\hline\\
At each node $i$: \\ \\
$local\_max \leftarrow x(i)$ \\
broadcast $local\_max$ to neighbors \\
when received $y$: \\
\hspace{1cm} if $(y > local\_max)$ \\
\hspace{1.2cm} $local\_max \leftarrow y$ \\
\hspace{1.2cm} broadcast $local\_max$ to neighbors \\
\hspace{1cm} endif\\
\hline
\end{tabular}
\caption{Distributed algorithm to compute $\max(x)$, where $x$ is a vector and $x(i)$ is stored at node $i$. The variable $local\_max$ at each node will eventually converge to $\max(x)$.}
\label{Tab:max}
\end{table}

\subsection{Computing spanning tree}
\label{subsec:computing_spanning_tree}
We start by selecting a root node. This is an arbitrary choice and we may choose, for example, the node with the maximum index. This node can be identified by algorithm given in Table \ref{Tab:max}. The root node initiates a broadcast packet, and all other nodes broadcast this packet to their neighbors upon reception. A node $v_i$ will choose a parent node $v_j$ in its neighboring set, such that the packet with least hop-count received by $v_i$ was relayed by $v_j$. Ties are broken arbitrarily. The pseudocode for this algorithm is shown in Table \ref{tab:spanning_tree}. At the end of the algorithm, each node knows its parent, its children and the hop length to the root node. Note that if the order in which a node receives packets from various paths is same as the hop-lengths of these paths, each node will broadcast the packet only once.

\begin{table}
\centering
\begin{tabular}{l}
\hline\\
Algorithm to compute spanning tree\\
\hline\\
At root node $v_r$\\
$hop\_count = 1$\\
broadcast $[v_r \mbox{  } hop_count]$ to $\mathcal{N}_r$ \\ \\

At any other node $v_i$ \\
$hop\_count \leftarrow \infty$ \\
$parent \leftarrow \varnothing$ \\
when received $[v_j \mbox{  } h]$ \\
\hspace{1cm} if $(h < hop\_count)$ \\
\hspace{1.2cm} $hop\_count \leftarrow h$ \\
\hspace{1.2cm} broadcast $[v_i \mbox{  } hop\_count+1]$ to $\mathcal{N}_i$ \\
\hspace{1.2cm} $parent \leftarrow v_j$ \\
\hspace{1.2cm} transmit ``$v_i$ is a child'' to $v_j$\\
\hspace{1cm} endif \\
\hline
\end{tabular}
\caption{a distributed algorithm to compute a spanning tree. When the algorithm terminates, each node will have a unique parent.}
\label{tab:spanning_tree}
\end{table}

\subsection{Identifying contractible cycles}
Theorem \ref{theo:contractible_cycles} states that a cycle $c$ is contractible (with probability 1) if and only if the integral $\left<y,c\right>$ of a harmonic $y$ on $c$ is zero. Note also that we are only interested in cycles in $Z$. The spanning tree can be effectively utilized to compute these integrals efficiently. We do this in two steps
\begin{enumerate}
\item compute an integral function $f:V\rightarrow \mathbb{R}$ on the nodes such that $f(v_i) = \left<y,\pi(p_i)\right>$, where $p_i$ is the path in $T$ joining the root node $v_r$ to $v_i$.
\item for a cycle $\gamma(T,e) \in Z$, $e = (v_j,v_k)$, the integral $\left<y,\gamma(T,e)\right>$ is equal to $f(v_j) + \left<y,e\right> - f(v_k)$
\end{enumerate}
Table \ref{tab:distributed_integral_computation} describes the algorithm to compute the integral function $f$. The root node initiates a broadcast which travels down the tree while computing the integral for each node. Each node broadcasts precisely once.   Step 2 can be performed locally at one of the incident nodes for each edge. At the end of step 2, we have identified all the cycles in $Z$ as contractible or non-contractible, with probability 1.

\begin{table}[!p]
\centering
\begin{tabular}{l}
\hline\\
Algorithm for computing the integral function\\
\hline\\
At root node $v_r$:\\
$f(v_r) \leftarrow 0$ \\
broadcast $[v_r \mbox{  } 0]$ to $\mathcal{N}_r$ \\ \\

At any other node $v_i$: \\
when received $[v_j \mbox{  } x]$ \\
if $(v_j == parent)$ \\
\hspace{1cm} if $(v_i > v_j)$ \\
\hspace{1.2cm} $temp \leftarrow y\left((v_i,v_j)\right)$ \\
\hspace{1cm} else \\
\hspace{1.2cm} $temp \leftarrow -y\left((v_i,v_j)\right)$ \\
\hspace{1cm} endif \\
\hspace{1cm} $f(v_i) \leftarrow x + temp$ \\
\hspace{1cm} broadcast $[v_i \mbox{  } f(v_i)]$ to $\mathcal{N}_i$ \\
endif
\end{tabular}
\caption{After computing a harmonic $y$, and a spanning tree, this algorithm computes the integral function $f$ on the tree.}
\label{tab:distributed_integral_computation}
\end{table}

\begin{figure}[!h]
\centering
\subfigure[]{
\includegraphics[width=0.4\textwidth]{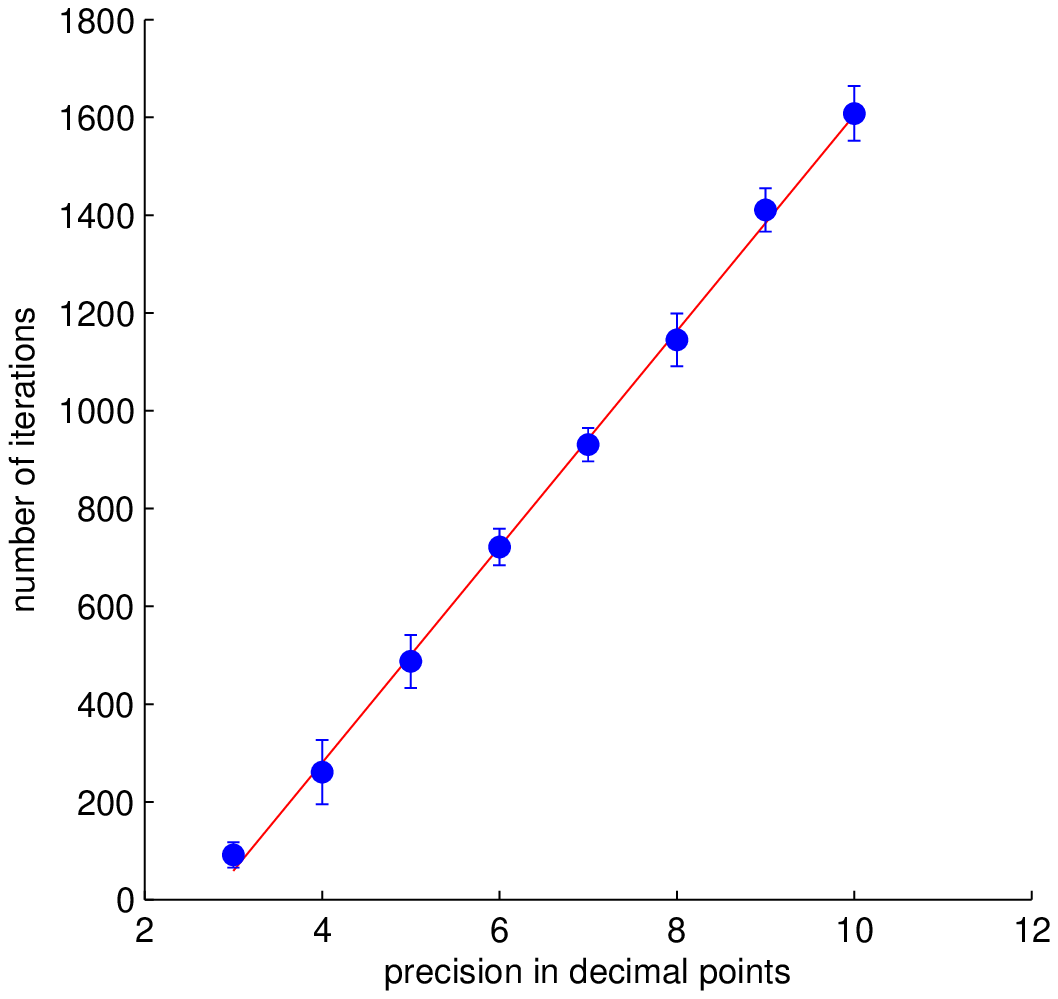}
}
\quad
\subfigure[]{
\includegraphics[width=0.4\textwidth]{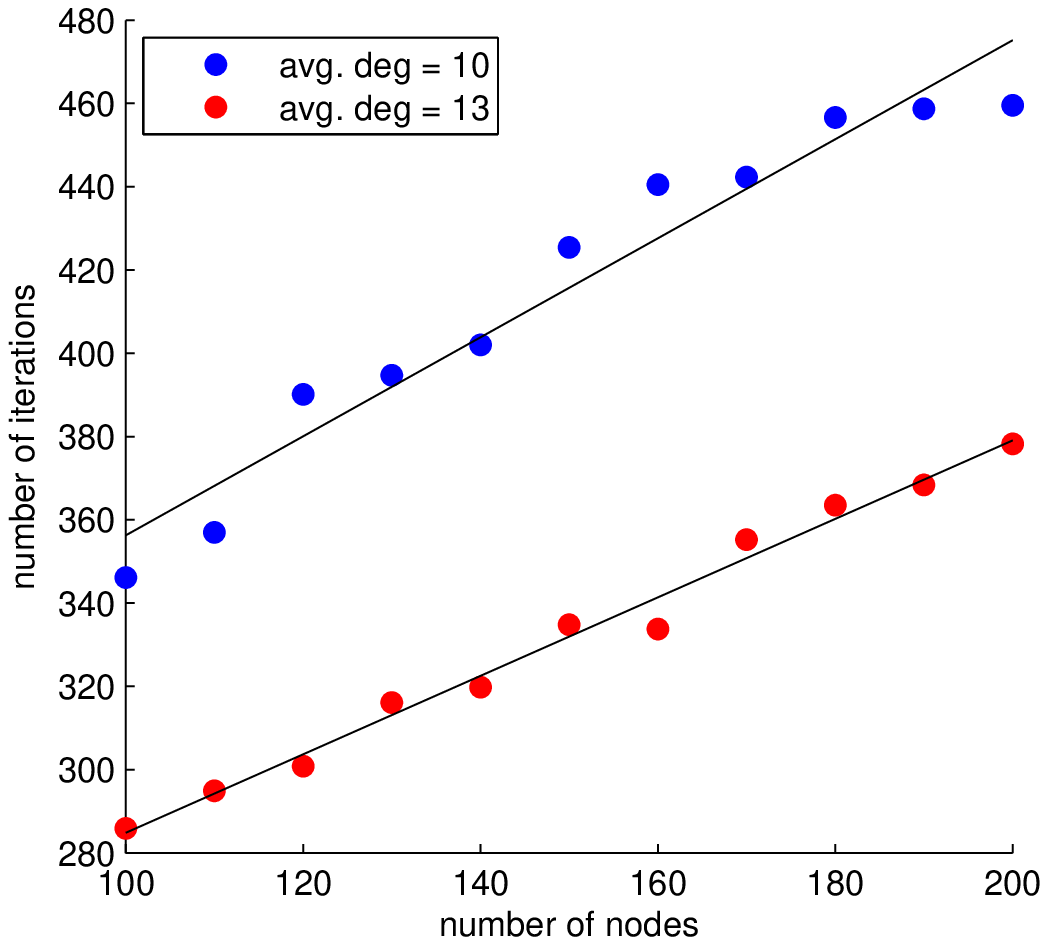}
\label{fig:linearIncreaseInIterations}
}
\caption{Figure (a) shows the average (over 100 realizations) number of iterations required to attain a required precision in decimal places. The regression line is shown in red, and the bars show the standard deviation. The linear relationship confirms exponential convergence as stated in Lemma \ref{theo:harmonicConvergence}. (b)shows the average (over 100 realizations) number of iterations to reach a precision of 6 decimal places as a function of the nodes for geometric graphs. As seen from this simulation, the growth is at most linear with the number of nodes.}
\end{figure}

\subsection{Selecting representative cycles for homologous cosets}
\label{subsec:selecting_P}
We first prune the tree of unnecessary nodes in order to reduce the communication complexity. The nodes which have been removed need not participate in the algorithm any further. If for a leaf node $v_l$ in $T$, there is no edge $e \in E\setminus E_T$ incident on $v_l$ such that $\gamma(e,T)$ is non-contractible, then we remove $v_l$ along with the edge joining $v_l$ to its parent node from $T$. We repeat this process until we cannot find any leaves satisfying the above condition. If the root node has to be removed, then its unique child will take its place as the root node. \\

Corollary \ref{cor:homologous_cycles} states that two cycles $c_1$ and $c_2$ are homologous if and only if the absolute values of the integrals $\left<y,c_1\right>$ and $\left<y,c_2\right>$ are equal. This absolute value can therefore be used as a label to identify a equivalence class of homologous cycle. The tree structure can then be exploited to select a cycle from each equivalence class efficiently. If the cycle $\gamma(e,T)$ corresponding to an edge $e = (v_i,v_j) \in E \setminus E_T$ is non-contractible, we say the nodes $v_i$ and $v_j$  are ``terminal nodes'' of cycle $\gamma(e,T)$. Note that all leaves are necessarily terminal nodes because of the pruning done above. \\

Each terminal node $v_i$ will send a packet containing terminal node pair $(v_i,v_j)$, the absolute value of the integral $\left<y,\gamma(e,T)\right>$ (which serves as a label) and the hop length of the cycle $\gamma(e,T)$ to its parent. Any non-terminal node, upon receiving the packets from all its children, clusters the cycles according to their labels and chooses a cycle with smallest hop-length from each cluster. It then transmits the information pertaining to the chosen cycles to its parent. When the root node performs the above computation, we have the required set $P$ of non-contractible cycles, no two of which are homologous to each other.  The maximum number of packets transmitted by any node remaining in the tree after pruning is equal to the cardinality of the set $P$.

\subsection{Reducing $P$ to obtain the homology generating set}
\label{subsec:reducing_P}
As described in Section \ref{subsec:homology_generating_subset}, reducing $P$ to select $H$ requires the computation of $|P|$ harmonics, where $|P|$ is the number of columns cycles in $P$. This can be performed by repeating the process described in Section \ref{subsec:distributed_computing_harmonics}, $|P|$ number of times. The harmonics are then integrated on the cycles in $|P|$ and transmitted to the root node for processing. The maximum number of packets any node has to transmit to send information to the root node is equal to $|P|^2$. Once the information is transmitted, the reduction is performed at the root node locally.

\section{Complexity}
\label{sec:complexity}
\subsection{Centralized computation}
As stated in Theorem \ref{theo:harmonicConvergence}, Iteration \ref{equ:randomHarmonicIteration} converges with exponentially with a rate $1-\lambda_1/\|L_1\|_1$ (here $\delta=1/\|L_1\|_1$). Let the required precision be $\epsilon$ and the number of iterations required to reach this precision be $\rho$. As per Equation \ref{iteration_equation}, the error term after the $k^{th}$ iteration is given as
\begin{equation}
\label{equ:lambda1_dominance}
\sum_i{\left(1-\delta\lambda_i\right)^k\alpha_iv_i} \approx \left(1-\delta\lambda_1\right)^k\alpha_1
\end{equation}
for precision $\epsilon$, we have
\begin{equation}
\epsilon = (1-\delta\lambda_1)^{\rho} |\alpha_1|
\end{equation}
As discussed in Appendix \ref{app:statistics_of_coefficients}, when the first Betti number $b_1$ is much smaller than $|E|$, then $\mathbf{E}\left[|\alpha|\right]$ approaches a constant value. The number of iterations $\rho$ required for a given precision $\epsilon$ hence varies as $O\left(\frac{\log\epsilon}{\log(1-\delta\lambda_1)}\right)$. \\

The computations required for each multiplication $L_1y^k$ is equal to the number of non-zero elements in $L_1$, which is usually sparse. Denote the number of non-zero elements in $L_1$ by $|L_1|$.  As discussed in Section \ref{subsec:reducing_P}, we compute $|P|$ number of harmonics. Therefore, the total complexity required for computing the harmonics is equal to $O\left(|P||L_1|\frac{\log\epsilon}{\log(1-\delta\lambda_1)}\right)$. Figure \ref{fig:linearIncreaseInIterations} shows simulation results for the number of iterations required as a function on nodes for geometric graphs. As seen from this simulation, the growth is at most linear with the number of nodes. \\

The steps 1) computation of spanning tree, 2) integrating a harmonic, 3) identifying the contractible cycles, 4)partitioning non-contractible cycles and 5)selecting a representative may all be accomplished with $O(E)$ complexity. \\

For the final step of reducing $P$ to homology generating set $H$, we need to integrate $|P|$ number of harmonics on each of the cycle in $P$. Since every cycle $c$ in $P$ is simple, the number of edges in $c$ is upper bounded by the number of nodes $N$ \footnote[2]{note that this is usually a loose upper-bound}. Integrating the harmonics to form the matrix $R$ hence requires $O\left(N|P|\right)$ computations. Finally, reducing the matrix $R$ has a complexity $O\left(|P|^\omega\right)$, where $\omega$ is the complexity order of matrix multiplication.\\

The complexity of the centralized algorithm as a whole is therefore given by
$O\left(|P||L_1|\frac{\log\epsilon}{\log(1-\delta\lambda_1)}\right)$ + $O(E)$ + $O\left(N|P|\right)$ + $O\left(|P|^\omega\right)$. The last term is usually the dominant factor. Note that $|P|$ is very small compared to the number of edges. In the worst case, $|P|=|E|$, and the complexity is comparable to other known algorithms for computing persistence \cite{milosavljevic2010zigzag}.  \\

If the original complex is contractible, $|P|$ is equal to zero as all the cycles in $Z$ will be contractible. In this case, the algorithm ends after computing the first harmonic and checking for non-contractible cycles, and  the complexity reduces to $O\left(|L_1|\frac{\log\epsilon}{\log(1-\delta\lambda_1)}\right)$ + $O(E)$. \\

In the case of geometric graphs in the critical regime\footnote[3]{where giant components begin to emerge}, $\mathbf{E}\left[|L_1|\right] = 2k(k-1/4)N$, where $k$ is the average node degree (see Appendix \ref{app:expectation_of_L_1}). Further, assuming $|P| \approx b_1$ as supported by Figure \ref{fig:linearExcessCycles}, the complexity simplifies to $O\left(b_1N\frac{\log\epsilon}{\log(1-\delta\lambda_1)}\right)$ +  $O\left(b_1^\omega\right)$.

\subsection{Distributed computation}
In distributed computation scenarios, the cost of communication between processing nodes is usually much higher than that of computation within nodes. In this section, we focus on the complexity of communications required between the processing nodes. Further it is also more appropriate  in these scenarios to analyze the cost per node. We also focus our analysis on systems with architecture similar to that of sensor networks, where each node has the capability to send information to all its neighbors simultaneously, as in the broadcasting.   \\

For computing the harmonics using Iteration \ref{equ:randomHarmonicIteration}, the number of packets each node has to broadcast (for edge being simulated) is equal to the number of iterations. As the edges are simulated on one of the incident nodes, a node $v_i$ will be simulating a maximum of $d_i$ number of edges. The average number of packets transmitted per node for one iteration is therefore $k$, the average of $d_i$.  The size of a packet corresponding to edge $e_i$ is proportional to the number of non-zero elements in $i^{th}$ row of $L_1$, the average of which we denote by $|L_1|_{avg}$.   Therefore, the average communication complexity per node for computing $|P|$ number of harmonics is given as $O\left(k|P||L_1|_{avg}\frac{log\epsilon}{log(1-\delta\lambda_1)}\right)$. \\

Computing the spanning tree, integrating harmonic along the tree, identifying contractible cycles and pruning the tree, each of which can be accomplished by a single broadcast from each node. As discussed in Section \ref{subsec:selecting_P}, selecting a representative cycle from each set of homologous cycles requires, in the worst case, for a node to make $|P|$ number of broadcasts. Integrating $|P|$ number of harmonics and transmitting them to the root node also require $|P|$ number of broadcasts. The communication complexity is dominated by the part of the algorithm computing the harmonics, and is equal to $O\left(k|P||L_1|_{avg}\frac{log\epsilon}{log(1-\delta\lambda_1)}\right), |P| \ll |E| $.

\section{conclusion}
\label{sec:conclusion}

We present a distributed algorithm to compute first homology given the 2-skeleton of a complex. This is made possible by 1) distributed computation of a basis for 1-cycles using spanning trees and 2) efficiently and locally identifying contractible and homologous cycles using harmonics. As discussed in Section \ref{sec:distributed_computation}, the spanning tree may be obtained very easily with a simple distributed algorithm with each node broadcasting a constant number of packets.  We show in Section \ref{sec:computing_harmonics} and \ref{sec:distributed_computation} that the harmonics may be computed with a simple distributed algorithm with exponential convergence.\\

Many algorithms for homology computation obtain a space which is isomorphic to the homology. In addition, we obtain explicit cycles, and since the coefficients are in $\mathbb{Z}_2$, are localized, i.e., have small number of non-zero coefficients compared to the number of edges. The centralized version of the algorithm is also faster than other known algorithms which do not first reduce the complex. The complexity for centralized and distributed versions of the algorithm are derived in Section \ref{sec:complexity}, and as shown, the complexity of the algorithm is polynomial in the first Betti number for complexes representing sensor networks. \\

The work presented here may also be generalized to higher dimensions. Harmonics for higher dimensions can be computed distributively in a manner very similar to the one discussed here. If we can distributively compute an equivalent to the spanning tree in higher dimension, which is the topic of our future research, then we may readily generalize the current procedure to higher dimensions.

\appendix

\section{Formula for elements of $L_1$}
\label{app:formulaForL1}
We need to introduce a few definitions before presenting the formula. Denote by $\sigma_i^j$, a simplex of dimension $j$ with index $i$. Two simplices $\sigma_1^j$, $\sigma_2^j$ are said to be upper adjacent, denoted $\sigma_1^j \frown \sigma_2^j$, is they are faces of a common simplex $\sigma_1^{j+1}$. Two simplices $\sigma_1^j$, $\sigma_2^j$ are said to be lower adjacent, denoted $\sigma_1^j \smile \sigma_2^j$, if they share a common face of dimension $j-1$. Two simplices $\sigma_1^j$, $\sigma_2^j$ are said to be similarly or dissimilarly oriented if they are lower adjacent, and they induce the similar or opposite orientation respectively on their common face. The upper degree of a simplex $\sigma_i^j$, denoted $deg_u(\sigma_i^j)$ is equal to the number of simplices of dimension $j+1$ with $\sigma_i^j$ as a face. The elements of $L_1:C_1\rightarrow C_1$ are given as follows:
$$ \{L_1\}_{ij} = \left\{
\begin{array}{ll}
deg_u\left(\sigma_i\right) + 2 & \qquad i = j \\
1 & \qquad \sigma_i \not\frown \sigma_j, \sigma_i \smile \sigma_j \mbox{$\sigma_i$ and $\sigma_j$ are similarly oriented}\\
-1 & \qquad \sigma_i \not\frown \sigma_j, \sigma_i \smile \sigma_j \mbox{$\sigma_i$ and $\sigma_j$ are disimilarly oriented}\\
0 & \qquad \mbox{otherwise}
\end{array}
\right. $$

\section{Expectation of coefficients $\{\alpha_i\}$}
\label{app:statistics_of_coefficients}
The coefficients $\{\alpha_i\}$ in the Equation \ref{equ:lambda1_dominance} are random variables. From $y^0 = \sum_i{\alpha_iv_i}$, we have
\begin{equation}
\label{equ:parseval_look_alike}
\|y^0\|^2 = \left\|\sum_i{\alpha_iv_i}\right\|^2 = \sum_i{\alpha_i^2}
\end{equation}
The second equality results from the fact that the eigenvectors $\{v_i\}$ are mutually orthogonal. The orthogonality of eigenvectors also implies the coefficients $\{\alpha_i\}$ are mutually independent. The elements of $y^0$ are generated independently from  a uniform distribution on the interval $\left[-0.5,0.5\right]$. This implies that coefficients $\{\alpha_i\}$ are identically and independently distributed. From Equation \ref{equ:parseval_look_alike}, we have
\begin{eqnarray}
\mathbf{E}\left[\|y^0\|^2\right] &=& |E|\mathbf{E}\left[y^0(1)^2\right] \nonumber \\
&=& \mathbf{E}\left[\sum_i^m{\alpha_i^2}\right] = \sum_i^m{\mathbf{E}\left[\alpha_i^2\right]} = m\mathbf{E}\left[\alpha_1^2\right] \ge m\left(\mathbf{E}\left[|\alpha_1|\right]\right)^2
\end{eqnarray}
Since $(\cdot)^2$ is a convex function, the inequality in equation above follows from Jensen's inequality. This leads to
\begin{equation}
\mathbf{E}\left[|\alpha_1|\right] \le \sqrt{\frac{|E|}{m}\mathbf{E}\left[y^0(1)^2\right]}
\end{equation}
$m$ is the number of non-zero eigenvalues of $L_1$. Since the kernel of $L_1$ is isomorphic to the first homology, $m = |E|-b_1$ where $b_1$ is the first Betti number. When $b_1$ is much smaller than $|E|$ (which is usually the case), the above equation reduces to
\begin{equation}
\label{equ:cooefficients_are_constants}
\mathbf{E}\left[|\alpha_1|\right] \le \sqrt{\mathbf{E}\left[y^0(1)^2\right]} = c
\end{equation}

\section{Expectation of $|L_1|$ for geometric graphs}
\label{app:expectation_of_L_1}
In this section, we derive the expectation of number of non-zero elements of $L_1$, denoted as $|L_1|$, for the case of geometric graphs in the critical regime. The critical regime we discuss here is the case when the parameter $r$ using in constructing the geometric graph varies as $r \propto N^{-d}$. It is shown that these values for $r$ result in emergence of giant components \cite{penrose2003random}. It is also easy to see that for large values of $N$, these values of parameter result in a constant average node degree. Denote the average node degree to be $k$. \\

The element $\{L_1\}_{ij}$ is non-zero if and only if $e_i$ is adjacent to $e_j$.
The total number of non-zero components is hence given as
\begin{equation}
\label{equ:number_of_non_zero_components_1}
|L_1| = \sum_{i=1}^{|E|}{\left(|\mathcal{N}_{e_i}|+1\right)} = \sum_{i=1}^{|E|}{\left(|\mathcal{N}_{i_1}| + |\mathcal{N}_{i_2}| - 1\right)}, e_i = (v_{i_1},v_{i_2})
\end{equation}
When the summation is carried out over $E$, the number of times each node $v_i$ appears in the second summation is equal to its node degree $d_i$. Equation \ref{equ:number_of_non_zero_components_1} then leads to
\begin{equation}
\label{equ:number_of_non_zero_components_2}
|L_1| = \left(\sum_{i=1}^{N}{ \sum_{j=1}^{d_i}{|\mathcal{N}_i|}}\right)-|E| = \sum_{i=1}^{N}{d_i^2}-\frac{1}{2}\sum_{i=1}^{N}{d_i}
\end{equation}
For large $N$, the degree on a node may be approximated using the normal distribution with mean $k$ and variance $k^2$. The expectation of $|L_1|$ is hence given as
\begin{equation}
\label{equ:expectation_of_L_1}
\mathbf{E}\left[|L_1|\right] = \mathbf{E}\left[ \sum_{i=1}^{N}{ \left(d_i^2-\frac{1}{2}d_i \right) } \right] = \sum_{i=1}^N{\left(k^2+k^2-(1/2)k\right)} = 2k\left(k-1/4\right)N
\end{equation}

\bibliographystyle{plain}
\bibliography{../theBigBib}

\end{document}